\documentclass[a4paper,11pt]{amsart}
\usepackage{hyperref, latexsym}
\usepackage{enumerate}
\usepackage{graphicx}
\usepackage{color}
\usepackage{amsmath}
\usepackage{comment}

\theoremstyle{plain}
\newtheorem{theorem}{Theorem}[section]
\newtheorem{lemma}[theorem]{Lemma}
\newtheorem{corollary}[theorem]{Corollary}
\newtheorem{proposition}[theorem]{Proposition}
\theoremstyle{definition}

\theoremstyle{remark}

\begin{document}

\title[]
      {Integrable outer billiards and rigidity}

\date{5 October 2023}
\author{Misha Bialy}
\address{School of Mathematical Sciences, Raymond and Beverly Sackler Faculty of Exact Sciences, Tel Aviv University,
Israel} 
\email{bialy@tauex.tau.ac.il}

\thanks{MB is partially supported by ISF grant 580/20 and DFG grant MA-2565/7-1  within the Middle East Collaboration Program.}

%\subjclass[2000]{} 
%\keywords{}
\begin{abstract}

In the present paper, we introduce a new generating function for outer billiards in the plane. Using this generating function, we prove the following rigidity result: if the exterior of the smooth convex plane curve $\gamma$ of positive curvature is foliated by continuous curves which are invariant under outer billiard map, then the curve $\gamma$ must be an ellipse.
	In addition to the new generating function used in the proof, we also overcome the non-compactness of the phase space by finding suitable weights in the integral-geometric part of the proof. Thus, we reduce the result to the Blaschke-Santalo inequality.

\end{abstract}
\maketitle

%%%%%%%%%%%%%%%%%%%%%%%%%%%%%%%%%%%%%%%%%%%%%%%%%%%%%%%%%%%%%%%%%%%%%%%%%%

\section{Introduction}

In this paper, we address the question of integrability of outer billiards.
Throughout this paper, we consider a $C^2$-smooth strictly convex closed curve $\gamma$ in the plane of positive curvature. The outer billiard map $T$ acts in $\Omega$, the exterior of $\gamma$, as follows:

Given a point $A\in\Omega$, its image $T(A)$ is defined by the condition that the segment $[A,T(A)]$ is tangent to $\gamma$ exactly at the middle of the segment. The map $T$ is a symplectic diffeomorphism  
of $\Omega$ with respect to the standard symplectic form of the plane. Thus, $\Omega$ is the phase space of the outer billiard.

The model of outer billiards was introduced by B. Neumann in the late 1950s \cite{neumann} and even earlier in 1945 by M. Day \cite{day}. Thereafter, Jürgen Moser popularized the system in the 1970s as a toy model for celestial mechanics \cite{moser}\cite{moser1}. 

\begin{comment}
	Moser provided a KAM-type result for $T$, given that the billiard curve $\gamma$ is sufficiently smooth ($C^l, l\geq 334$). Thus, he proved the nonexistence of unbounded orbits in the $C^l$-smooth case. The regularity was improved to $l>6$ in \cite{douady} with the help of the result by H.~Russmann \cite{russmann}. It then follows from  M.~Herman's  \cite 
{herman} that $l>3$ is enough.
\end{comment}
%%%%%%%%%%%%%%%%%%%%%%
\begin{comment}

For a sufficiently large $r$, and a billiard curve $\gamma$ of class $C^{r}$, one can apply Moser twist mapping theorem \cite{moser1} in the class $C^{l},  l:=r-1$ (see \cite{douady} for a proof of this application). It then follows that there always exist invariant curves of the outer billiard arbitrary close to infinity, and therefore all orbits of the billiard are bounded. The initial requirement, $l\geq 333$, in Moser' theorem was improved by H.~Russmann \cite{russmann}, $l>5$, and then by M.~Herman's  \cite 
{herman} showing that $l>2$ is enough. 

\end{comment}
%%%%%%%%%%%%%%%%%%%%%%
Given a billiard curve $\gamma$ of class $C^r$, the billiard map $T$ is of class $C^l$, where $l=r-1$. If $l$ is sufficiently large, then the Moser twist mapping theorem \cite{moser1} applies (see \cite{douady} for a proof of this application). It then follows that there always exist invariant curves of the outer billiard arbitrary close to infinity, and therefore all orbits of the billiard are bounded. The initial requirement, $l\geq 333$, in Moser' theorem was improved by H.~Russmann \cite{russmann}, $l>5$, and then by M.~Herman's  \cite 
{herman} showing that $l>2$ is enough. 

Later, unbounded orbits for outer billiards were proven to exist for curves which are not $C^1$-smooth, starting from the work of R. Schwartz \cite{schwartz}.  In \cite{Tab}, unbounded orbits were discovered for the semi-circle by computer experiments, and were confirmed theoretically in \cite{bassam}. Apparently, nothing is known on this problem for the smoothness of $\gamma$ between $C^1$ and $C^r,r>3$.

 In this paper, we address the natural question which is analogous to Birkhoff-Poritsky conjecture for usual billiards (see \cite{K-S}\cite{G}\cite{bialy-mironov}\cite{koval} for recent progress):
 
{\it Are there integrable outer billiards in the plane other than ellipses?}
 
 For the algebraic version of this question, the answer is negative, as shown in \cite{T} \cite{glutsyuk-shustin}.

 It is easy to see from the definition of outer billiard of $\gamma$ that the group of affine transformations commutes with outer billiard map. Thus, the outer billiard map of the ellipse preserves the foliation of $\Omega$ by concentric homothetic ellipses (since this is obviously the case for a circle).

 \begin{figure}[h]\label{}
	\centering
	\includegraphics[width=0.8\linewidth]{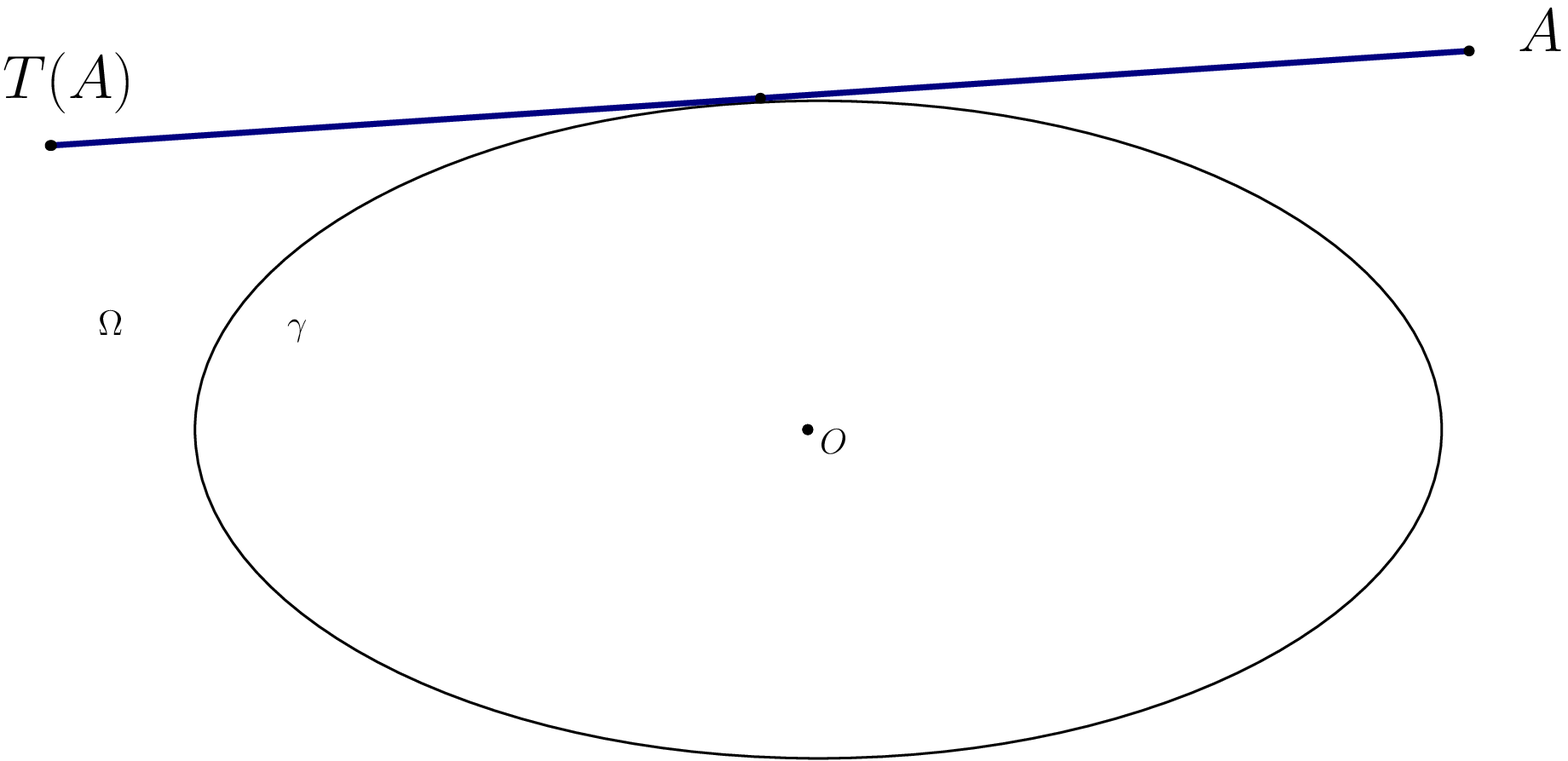}
	\caption{Outer billiard map}
	%\label{main}
\end{figure}

We now turn to formulate our main results.

\begin{theorem}\label{thm:total}
	Assume that the outer billiard of $\gamma $ is totally integrable, i.e., the phase space $\Omega$ is foliated by continuous rotational (i.e., non-contractable in $\Omega$) invariant curves, then $\gamma$ is an ellipse.
\end{theorem}

This result can be  stated in variational terms, as we now turn to explain. In Section 2, we shall introduce a non-standard generating function $S$ for the outer billiard, that corresponds to the symplectic polar coordinates $(p,\phi)$ on $\Omega$, where $p=r^2/2$, $r$ is the radial distance, and $\phi$ is the polar angle.  Then $dp\wedge d\phi$ is the standard symplectic form.
Moreover, we will prove that function $S$ satisfies the {\it positive} twist condition, $S_{12}(\phi_0,\phi_1)<0$ (Theorem \ref{thm:twist} below). 

Consider the corresponding action functional (see e.g. the survey \cite{siburg}):

\begin{equation}\label{functional}\{\phi_n\}  \rightarrow   \sum_{-\infty}^{+\infty}S(\phi_n,\phi_{n+1}).\end{equation}

The extremal configurations $\{\phi_n\}$ of (\ref{functional}) are in one-to-one correspondence with the orbits of the billiard map $T$, $\{(p_n,\phi_n)\}$, where
$$p_n=r_n^2/2=-S_1(\phi_n,\phi_{n+1})=S_2(\phi_{n-1},\phi_n).$$

We shall say that the extremal $\{\phi_n\}$ is  \emph {locally minimizing} if any finite sub-segment $\{\phi_n\}_{n=M}^N, M\leq N$ is a local minimum of the function
$$F_{M,N}(x_{M},...,x_{N})=S(\phi_{M-1},x_{M})+\sum_{i=M}^{N-1}S(x_{i},x_{i+1})+S(x_{N},\phi_{N+1}).$$We term the corresponding orbit $\{(p_n,\phi_n)\}$ as \emph{locally minimizing orbit}. 

We shall denote by $\mathcal M$ the subset of the phase space $\Omega$ swept by the locally minimizing orbits.
\begin{theorem}\label{thm:maximizing}
	If all orbits of the outer billiard $T$ are locally minimizing, then $\gamma$ is an ellipse.
	\end{theorem}
	Theorem \ref{thm:maximizing} implies the following geometric fact, which is equivalent to the existence of conjugate points for a billiard configuration.
	\begin{corollary}\label{conjugate}
		For any outer billiard, which is different from ellipse, there exists a radial tangent vector $v\in T_x\Omega, v=\frac{\partial}{\partial r}(x)$ and a positive $n$, such that after $n$ iterations the vector $DT^n(v)$ is radial again.
	\end{corollary}

Now, I will explain how Theorem  \ref{thm:total} follows from Theorem  \ref{thm:maximizing}. 
Given a rotational invariant curve of a positive twist symplectic map, then by Birkhoff theorem this curve is a graph of a Lipschitz function (see e.g.  \cite{siburg}).  Therefore, it is differentiable almost everywhere,  and hence almost every orbit $\{(p_n,q_n)\}$ on the curve has an invariant  tangent field $\{(\delta p_n,\delta q_n)\}$ with $\delta q_n>0$. It then follows from the criterion for local minimality (Theorem 1.1 of \cite{Bialy-Tsodikovich}) that almost every orbit on the invariant curve is locally minimizing. Since the set $\mathcal M$, which is swept by all locally minimizing orbits, is closed \cite{Bialy-Tsodikovich}, then {\it all} orbits on the rotational invariant curve are locally minimizing. Therefore, if there exists a foliation of $\Omega$ by rotational invariant curves, then all orbits in the phase space $\Omega$ are locally minimizing. Thus, Theorem  \ref{thm:total} follows from Theorem  \ref{thm:maximizing}. 

\smallskip

{\bf Remarks.}

	{\it 1. In fact M. Herman proved that all orbits on a rotational invariant curve are actually global minimizers.
	
	2. We will not use this in this paper, but in fact Theorem \ref{thm:total}  implies Theorem \ref{thm:maximizing}.
	Namely, if all orbits are locally minimizing, then one can  reconstruct the foliation by rotational invariant curves.  This was first performed by J. Heber \cite {heber} for geodesic flows, and then was extended to twist maps in \cite{cheng-sun},\cite{arcostanzo},\cite{arnaud}. 
	
	3. In \cite{Bialy-Tsodikovich}, we formulated the criterion for local maximality since  we considered there negative twist maps.}
\smallskip

The main idea in the proof of Theorem \ref{thm:maximizing} is to apply the so-called E. Hopf type rigidity for the case of billiard dynamics. 
This method has two parts. First, along every locally minimizing orbit one can construct a positive discrete Jacobi field and the corresponding auxiliary function $\omega$, which must satisfy certain evolution under $T$. 
The construction of the positive Jacobi field is the discrete analog of E. Hopf original
construction.
We refer to \cite{B0},\cite{Bialy-Tsodikovich}, and also Section 5 for the details. 

The second part of the method is to prove an integral geometric inequality, which can be consistent with the evolution of $\omega$ only for the case of ellipses.  

This rigidity method was first found in \cite{B0} for ordinary billiards. Later, it was realized for billiards on the sphere and hyperbolic plane \cite{B1} and also for magnetic billiards in \cite{B2}. Recently, it was also successfully applied in \cite{bialy-mironov} for Birkhoff-Poritsky conjecture in a centrally symmetric case. Interestingly, in the paper \cite{bialy-mironov}, the integral geometry part was performed with suitable weights. This is also the case here.

A new interesting class of symplectic billiards was introduced in \cite{albers-tabachnikov}. In a recent paper \cite{baracco-bernardi}, rigidity for this model of billiards was proven using new ideas.

The rigidity for outer billiard total integrability remained resistant for a long time due to two main difficulties. Firstly, the phase space $\Omega$ is not compact, which requires suitable weights for the integral geometric part of the proof. Secondly, the affine nature of the problem makes it harder than in the case of ordinary Birkhoff billiards.

In this paper, we present the correct weights and reduce the proof of the main theorems to the Blaschke-Santalo inequality of affine geometry. Crucial new tool in the proof is the non-standard generating function of outer billiards. This generating function is somewhat analogous to the one used in \cite{bialy-mironov} for usual billiards.

In Section 2 and 3, we construct the non-standard generating function $S$ and compute the derivatives for the change of variables. In Section 4, we find the second derivatives of $S$ and verify the twist condition.
In Section 5, we recall the properties of the function $\omega$ and its evolution.
In Section 6, we obtain an inequality which is valid under the assumption that all orbits are locally minimizing. In Section 7, we show that in fact, the converse inequality holds, provided the origin is chosen at the Santalo point of the curve $\gamma$. Finally, we prove Theorem \ref{thm:maximizing} and the Corollary at Section 8.

\section*{Acknowledgments}

I discussed the problem of integrable outer billiards with Sergei Tabachnikov for many years. I am thankful to him for illuminating discussions and for the encouragement. 

I would also like to thank Luca Baracco, Olga Bernardi, Yaron Ostrover, and Leva Buhovsky for interesting and helpful discussions. I am grateful to Oleg Shaynkman for his help in conducting computer simulations. 

I am grateful to Marie-Claude Arnaud for giving the reference to M.Herman result, and to the anonymous referees for careful reading and suggestions for improvement.

\section{Non-standard generating function}

We introduce a new generating function $S$ which is different from the one used in the previous papers, for example, from the one used in 
\cite{boyland}\cite{douady}\cite{gk}\cite{Tab} and the book \cite{tabachnikov}.

Let us fix the origin $O$ to be an arbitrary point inside the curve $\gamma$. Later, in Section 7, we will need to specify the origin to be at the Santalo point of the convex body $\Gamma$, which is bounded by $\gamma$.

Take a standard symplectic form in the plane and use the symplectic polar coordinates with respect to $O$: $$dx\wedge dy=dp\wedge d\phi =rdr\wedge d\phi, p:=r^2/2.$$

For the billiard map $$T:(p_0,\phi_0)\mapsto(p_1,\phi_1), \ p_i=r_i^2/2,\  i=0,1,$$ we wish to find the generating function corresponding to the primitive 1-form $\alpha=pd\phi=(r^2/2)d\phi$. Therefore, we need to find a function $S(\phi_0,\phi_1)$ depending on the two angles such that
\begin{equation}\label {generatingS}
	T^*\alpha-\alpha=dS\quad  \Leftrightarrow   \quad S_1=-p_0=-r_0^2/2, \ S_2=p_1=r_1^2/2.
	\end{equation}

 Here and below, we shall use sub-indices 1 and 2 of $S$ for the partial derivatives with respect to $\phi_0,\phi_1$ respectively.

The function $S(\phi_0,\phi_1)$ can be easily found from geometric considerations, but we will also give the computational proof below. Given the values $\phi_0, \phi_1$,  $\phi_0<\phi_1<\phi_0+\pi$, consider the segment with the ends lying on the rays with the angles $\phi_0$ and $\phi_1$, which is tangent to the curve $\gamma$ exactly at the middle (Figure \ref{generating}). 
One can easily see that the point $M$ is the tangency point of $\gamma$ with a
hyperbola whose asymptotes are the rays in the directions $\phi_0$ and $\phi_1$ (the envelope
of the lines that cut off a fixed area from an angle is a hyperbola whose asymptotes are
the sides of the angle).

\begin{figure}[h]\label{generating}
	\centering
	\includegraphics[width=0.7\linewidth]{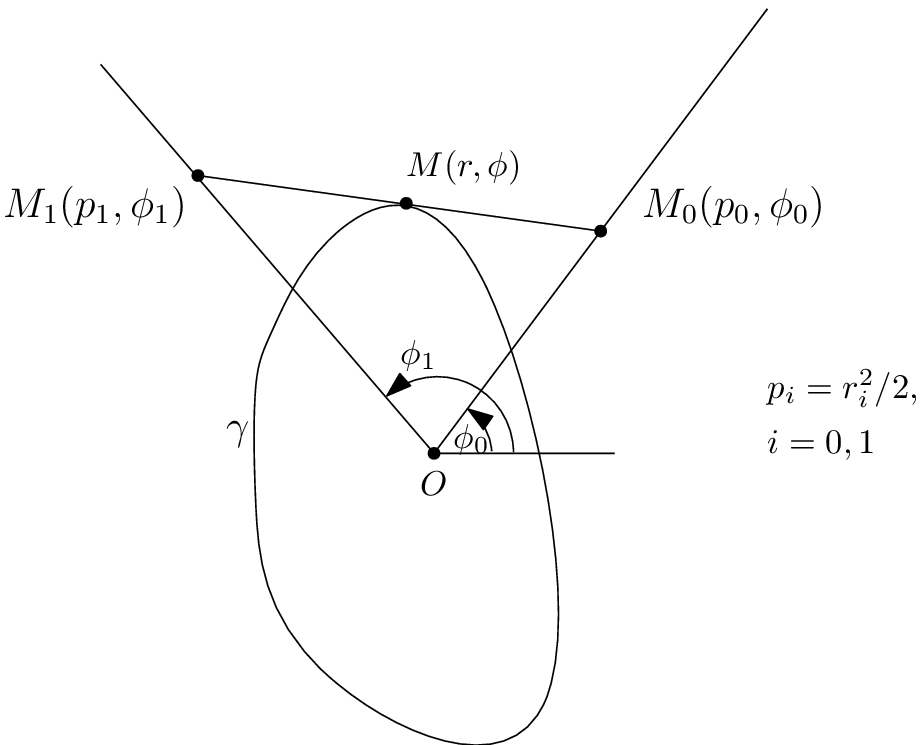}
	\caption{Generating function $S$; The point $M$ is the tangency point of $\gamma$ with a
		hyperbola whose asymptotes are the rays in the directions $\phi_0$ and $\phi_1$}
	%\label{main}
\end{figure}
Then, we define $S(\phi_0,\phi_1)$ to be the area of the triangle $\triangle M_0OM_1$ bounded by the two rays and the segment. 
 
 With this definition of function $S$, Figure \ref{Sderivative} gives a geometric explanation of (\ref{generatingS}).
 Indeed, let us consider $\tilde\phi_1:=\phi_1+\epsilon$ and compute the difference:
 $$
 \delta:=S(\phi_0,\phi_1+\epsilon)-S(\phi_0,\phi_1)=Area(\triangle \tilde M_0O\tilde M_1)-Area( \triangle M_0OM_1)
 $$

 Then we have (see Figure \ref{Sderivative}):
 \begin{figure}[h]
 	\centering
 	\includegraphics[width=0.7\linewidth]{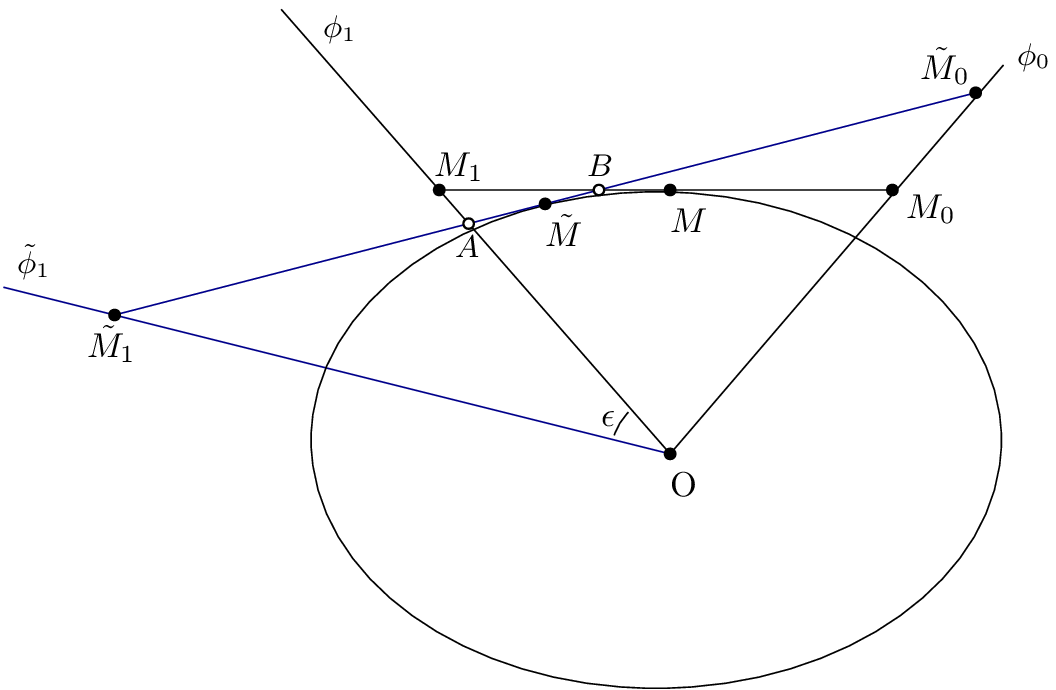}
 	\caption{Geometric derivation of $S_2=r_1^2$.}
 	\label{Sderivative} 
 \end{figure}
  $$
 \delta=Area(\triangle AO\tilde M_1)+Area(\triangle M_0B\tilde M_0)-Area(\triangle ABM_1)
 $$
 Moreover, we have $$
 Area(\triangle M_0B\tilde M_0)-Area(\triangle ABM_1)=o(\epsilon),
 $$
 since $M$ is the midpoint of the segment $[M_0,M_1]$.
 
  Therefore, $$\frac {\partial S(\phi_0,\phi_1)}{\partial\phi_1}=\lim_{\epsilon\rightarrow 0}\frac{\delta}{\epsilon}=\lim_{\epsilon\rightarrow 0}\frac{Area(\triangle AO\tilde M_1)}{\epsilon}=\frac{|OM_1|^2}{2}= \frac{r_1^2}{2},$$ as required in (\ref{generatingS}).

We shall compute now $S$ as follows. Let $\gamma$ be parameterized by the polar angle $\phi$, and the radial function of $\gamma$ will be denoted by $r(\phi)$. Hence, $$\gamma(\phi)=r(\phi)e_\phi, \quad \gamma'(\phi)=r'(\phi)e_\phi+r(\phi)e^\perp_\phi,$$
where $e_\phi, e^\perp_\phi$ are the unit vectors in the directions $\phi, \phi +\pi/2$. 
We shall write:
$$ M=\gamma(\phi),\ M_0= \gamma(\phi)-t\gamma'(\phi), \  M_1= \gamma(\phi)+t\gamma'(\phi).$$
With these formulas, $S$ looks especially simple:
$$S=tr^2(\phi).$$

\section {Derivatives for the change of variables}
In addition, we have the explicit formulas for the transition from $(\phi,t)$ to $(\phi_0,\phi_1)$ and for $r_0,r_1$.

$$\phi_0=\phi-\arctan\left(\frac{tr}{r-tr'}\right),\ \phi_1=\phi+\arctan\left(\frac{tr}{r+tr'}\right)$$
$$r_0^2=(r(\phi)-tr'(\phi))^2+t^2r(\phi)^2,\ \  r_1^2=(r(\phi)+tr'(\phi))^2+t^2r(\phi)^2.
$$

We shall introduce the following notation:
$$
\chi:=r^2(\phi)+2r'^2(\phi)-r(\phi)r''(\phi),
$$
which is the numerator of the formula for the curvature $k$ of $\gamma$ in the polar coordinates, $$k=\frac{\chi}{(r^2+r'^2)^{\frac{3}{2}}},$$ and hence is strictly positive.

The Jacobian matrix $J=\frac{\partial(\phi_0,\phi_1)}{\partial(\phi,t)}$ and the inverse can be easily computed.

\begin{flalign*}
&	\partial_{\phi}\phi_0=1-\frac{t^2 \left(-\chi+\mathit{r}(\phi )^2+\mathit{r}'(\phi )^2\right)}{\mathit{r}(\phi )^2+t^2 \mathit{r}(\phi )^2-2 t \mathit{r}(\phi
		) \mathit{r}'(\phi )+t^2 \mathit{r}'(\phi )^2}=1-\frac{t^2 \left(-\chi+\mathit{r}^2+\mathit{r}'^2\right)}{r_0^2}\\
&
\partial_{t}\phi_0=-\frac{\mathit{r}(\phi )^2}{\mathit{r}(\phi )^2+t^2 \mathit{r}(\phi )^2-2 t \mathit{r}(\phi ) \mathit{r}'(\phi )+t^2 \mathit{r}'(\phi
	)^2}=-\frac{r^2}{r_0^2},\\
&
\partial_{\phi}\phi_1=1-\frac{t^2 \left(-\chi+\mathit{r}(\phi )^2+\mathit{r}'(\phi )^2\right)}{\mathit{r}(\phi )^2+t^2 \mathit{r}(\phi )^2+2 t \mathit{r}(\phi
	) \mathit{r}'(\phi )+t^2 \mathit{r}'(\phi )^2}=1-\frac{t^2 \left(-\chi+\mathit{r}^2+\mathit{r}'^2\right)}{r_1^2}\\
&
		\partial_{t}\phi_1=\frac{\mathit{r}(\phi )^2}{\mathit{r}(\phi )^2+t^2 \mathit{r}(\phi )^2+2 t \mathit{r}(\phi ) \mathit{r}'(\phi )+t^2 \mathit{r}'(\phi )^2}=\frac{r^2}{r_1^2}.
\end{flalign*}

The determinant of $J$ reads:

\begin{flalign}
	|J|=&\\ \nonumber
	=&\frac{2 \mathit{r}(\phi )^2 \left(\chi t^2+\mathit{r}(\phi )^2\right)}{\left(\mathit{r}(\phi )^2+t^2 \mathit{r}(\phi )^2-2 t \mathit{r}(\phi
		) \mathit{r}'(\phi )+t^2 \mathit{r}'(\phi )^2\right) \left(\mathit{r}(\phi )^2+t^2 \mathit{r}(\phi )^2+2 t \mathit{r}(\phi ) \mathit{r}'(\phi )+t^2
		\mathit{r}'(\phi )^2\right)}=\\  \nonumber
		=&\frac{2 \mathit{r}(\phi )^2 \left(\chi t^2+\mathit{r}(\phi )^2\right)}{r^2_0r^2_1}
		\end{flalign}

From here we know the inverse matrix $J^{-1}=\frac{\partial (\phi,t)}{\partial(\phi_0,\phi_1)}$.
We have:

\begin{flalign}\label{dphidphi1}
\partial_{\phi_0}\phi&=\frac{1}{|J|}\partial_t\phi_1=\frac{\mathit{r}(\phi )^2+t^2 \mathit{r}(\phi )^2-2 t \mathit{r}(\phi ) \mathit{r}'(\phi )+t^2 \mathit{r}'(\phi )^2}{2 \left(\chi t^2+\mathit{r}(\phi
	)^2\right)}=\frac{r_0^2}{2 \left(\chi t^2+\mathit{r}^2\right)},\\ \nonumber
\partial_{\phi_1}\phi&=-\frac{1}{|J|}\partial_t\phi_0=\frac{\mathit{r}(\phi )^2+t^2 \mathit{r}(\phi )^2+2 t \mathit{r}(\phi ) \mathit{r}'(\phi )+t^2 \mathit{r}'(\phi )^2}{2 \left(\chi t^2+\mathit{r}(\phi
	)^2\right)}=\frac{r_1^2}{2 \left(\chi t^2+\mathit{r}^2\right)},\\ \nonumber
\partial_{\phi_0}t&=-\frac{1}{|J|}\partial_\phi\phi_1=-\frac{\left(\chi t^2+\mathit{r}(\phi )^2+2 t \mathit{r}(\phi ) \mathit{r}'(\phi )\right) \left(\mathit{r}(\phi )^2+t^2 \mathit{r}(\phi )^2-2
	t \mathit{r}(\phi ) \mathit{r}'(\phi )+t^2 \mathit{r}'(\phi )^2\right)}{2 \mathit{r}(\phi )^2 \left(\chi t^2+\mathit{r}(\phi )^2\right)}=\\ \nonumber
&=-\frac{r_0^2\left(\chi t^2+\mathit{r}^2+2 t \mathit{r} \mathit{r}'\right) }{2 \mathit{r}^2 \left(\chi t^2+\mathit{r}^2\right)},\\ \nonumber
\partial_t\phi_1&=\frac{1}{|J|}\partial_{\phi}\phi_0=\frac{\left(\chi t^2+\mathit{r}(\phi )^2-2 t \mathit{r}(\phi ) \mathit{r}'(\phi )\right) \left(\mathit{r}(\phi )^2+t^2 \mathit{r}(\phi )^2+2
	t \mathit{r}(\phi ) \mathit{r}'(\phi )+t^2 \mathit{r}'(\phi )^2\right)}{2 \mathit{r}(\phi )^2 \left(\chi t^2+\mathit{r}(\phi )^2\right)}=\\ \nonumber
&=\frac{r_1^2\left(\chi t^2+\mathit{r}^2-2 t \mathit{r}\mathit{r}'\right)}{2 \mathit{r}^2 \left(\chi t^2+\mathit{r}^2\right)}.
\end{flalign}

Now, we can confirm the geometric conclusion that $S$ is a generating function by the exact computation. Namely, from $S=r^2t$ and formulas (\ref{dphidphi1}) we have:

$$S_1= S_\phi \partial_{\phi_0}(\phi)+S_t \partial_{\phi_0}(t)=2rr't\frac{r_0^2}{2 \left(\chi t^2+\mathit{r}^2\right)}+r^2\left(-\frac{r_0^2\left(\chi t^2+\mathit{r}^2+2 t \mathit{r} \mathit{r}'\right) }{2 \mathit{r}^2 \left(\chi t^2+\mathit{r}^2\right)}\right)=-\frac{r_0^2}{2},
$$
$$
S_2=S_\phi \partial_{\phi_1}(\phi)+S_t \partial_{\phi_1}(t)=2rr't\frac{r_1^2}{2 \left(\chi t^2+\mathit{r}^2\right)}+r^2\left(\frac{r_1^2\left(\chi t^2+\mathit{r}^2-2 t \mathit{r}\mathit{r}'\right)}{2 \mathit{r}^2 \left(\chi t^2+\mathit{r}^2\right)}\right)=\frac{r_1^2}{2}.
$$

Thus, we have proved:
\begin{theorem}\label{thm:generating}
The function $S(\phi_0,\phi_1)$, which equals the area of the triangle $\triangle M_0OM_1$ (Fig. \ref{generating}), is a generating 
function of the outer billiard map with respect to the symplectic polar coordinates in the plane.
\end{theorem}

\section{Partial derivatives of the generating function \emph{S}}

Finally, using the formulas (\ref {dphidphi1}), we can find by the chain rule the derivatives of the function $S$ with respect to $\phi_0,\phi_1$ denoted by sub-indices $1,2$. 
It is important to mention that all the formulas are rational functions in $t$ with coefficients depending on $r(\phi),r'(\phi),r''(\phi)$.

\begin{flalign}\label{Sderivatives}
S_{11}&=\frac{\left(\left(1+t^2\right) \mathit{r}^2-2 t \mathit{r} \mathit{r}'+t^2 \mathit{r}'^2\right) \left(\chi t
	\left(-1+t^2\right) \mathit{r}^2+2 t \mathit{r}^4-2 \mathit{r}^3 \mathit{r}'+t \left(\chi t^2+2 \mathit{r}^2\right)
	\mathit{r}'^2\right)}{2 \mathit{r}^2 \left(\chi t^2+\mathit{r}^2\right)}=\\ \nonumber
&=\frac{r_0^2\left(\chi t
	\left(-1+t^2\right) \mathit{r}^2+2 t \mathit{r}^4-2 \mathit{r}^3 \mathit{r}'+t \left(\chi t^2+2 \mathit{r}^2\right)
	\mathit{r}'^2\right)}{2 \mathit{r}^2 \left(\chi t^2+\mathit{r}^2\right)},\\ \nonumber
S_{22}&=\frac{\left(\left(1+t^2\right) \mathit{r}^2+2 t \mathit{r} \mathit{r}'+t^2 \mathit{r}'^2\right) \left(\chi t
	\left(-1+t^2\right) \mathit{r}^2+2 t \mathit{r}^4+2 \mathit{r}^3 \mathit{r}'+t \left(\chi t^2+2 \mathit{r}^2\right)
	\mathit{r}'^2\right)}{2 \mathit{r}^2 \left(\chi t^2+\mathit{r}^2\right)}=\\ \nonumber
&=\frac{r_1^2 \left(\chi t
	\left(-1+t^2\right) \mathit{r}^2+2 t \mathit{r}^4+2 \mathit{r}^3 \mathit{r}'+t \left(\chi t^2+2 \mathit{r}^2\right)
	\mathit{r}'^2\right)}{2 \mathit{r}^2 \left(\chi t^2+\mathit{r}^2\right)},\\ \nonumber
S_{12}&=\frac{\chi t \left(-\left(1+t^2\right)^2 \mathit{r}^4-2 t^2 \left(-1+t^2\right) \mathit{r}^2 \mathit{r}'^2-t^4 \mathit{r}'^4\right)}{2 \mathit{r}^2 \left(\chi t^2+\mathit{r}^2\right)}=-\frac{\chi t r_0^2r_1^2}{2 \mathit{r}^2 \left(\chi t^2+\mathit{r}^2\right)}.
\end{flalign}

As a consequence of the last formula of (\ref{Sderivatives}), the theorems of Birkhoff and Herman apply, and we get the following:
\begin{theorem}\label{thm:twist}
The cross derivative of $S_{12}(\phi_0,\phi_1)$ is strictly negative on $\Omega$ and hence the outer billiard map $T$ is a {\it positive} twist map with respect to polar coordinates.

Every continuous rotational invariant curve of $T$ is star-shaped and is Lipschitz (the radial function $r(\phi)$ is Lipschitz). 

Moreover, all orbits on the invariant curve are minimizers for the functional (\ref{functional}).
\end{theorem}

\section{The inequalities for total rigidity}

Let $A$ be a positive function on the phase space $\Omega$ of outer billiard. We shall denote by $B:=A\circ T$.
\begin{theorem}\label{thm:inequality}Assume that all orbits of an outer billiard are locally minimizing.  Then, for any positive $A$ and $B=A\circ T$ we have the following inequality, provided that the integral $I$ is absolutely converging on $\Omega$:
	$$
I:=	\int_{\Omega}[A^2S_{11}+2ABS_{12}+B^2S_{22}]d\mu\geq0,
	$$
	where the invariant measure $d\mu$ is given by $$d\mu= dxdy=-S_{12}d\phi_0d\phi_1=-S_{12}|J|d\phi dt.$$
	\end{theorem}

\begin{proof}

	The first step of the proof is valid for any twist symplectic map of the cylinder.  
	We refer to the papers \cite{B0}  and \cite{Bialy-Tsodikovich} for the ideas of the method and the details. 
	
	Let $T$ be a symplectic positive twist map of the cylinder with the coordinates $(p,q)$, having a generating function $S(q_0,q_1)$. The positive twist condition reads $S_{12}<0$.
	
	A \textit{discrete Jacobi field} along a configuration $\{q_n\}$ is a sequence $\{\delta q_n\}$ satisfying the
	\textit{discrete Jacobi equation}:
	\begin{equation}
		\label{eq:Jacobi}
		b_{n-1}\delta q_{n-1}+a_n
		\delta q_n+b_{n}\delta q_{n+1}=0,
	\end{equation}
	where $a_n:=S_{22}(q_{n-1},q_n)+S_{11}(q_n,q_{n+1}),\
	b_n:=S_{12}(q_n,q_{n+1}).$

	For every locally minimizing orbit $\{(p_n,q_n)\}$, we can construct a strictly positive discrete Jacobi field $\{\delta q_n\}$ along the configuration $\{q_n\}$ normalized by $\delta q_0=1$ (see \cite{B0} and \cite{Bialy-Tsodikovich}). Moreover, by the very construction,  this Jacobi field depends measurably on the initial point $(p_0,q_0)$. In addition, the Jacobi field $\{\delta q_n\}$ naturally corresponds to an invariant non-vertical vector-field $\{(\delta p_n,\delta q_n)\}$ along the orbit $(p_n,q_n)$ with $\delta q_n>0$.

\begin{equation}\label{eq:deltap1}
		\delta p_n=-S_{11}(q_n,q_{n+1})\delta q_n-S_{12}(q_n,q_{n+1})\delta q_{n+1},
	\end{equation}
	or equivalently, due to the equation (\ref{eq:Jacobi}) on Jacobi fields:
\begin{equation}\label{deltap2}
	\delta p_n=S_{22}(q_{n-1},q_{n})\delta q_n+S_{12}(q_{n-1},q_{n})\delta q_{n-1}.
\end{equation}
	Set $${\omega(p_n,q_n):=\frac{\delta p_n}{\delta q_n}}.$$ Then $\omega$ is a measurable function and satisfies the evolution relations:
	\begin{equation}\label{relations}
	\begin{cases}
		\omega(T(p_0,q_0))=S_{22}(q_0,q_1)+S_{12}(q_0,q_1)\delta q_1(p_0,q_0)^{-1},\\
		\omega(p_0, q_0)=-S_{11}(q_0,q_1)-S_{12}(q_0,q_1)\delta q_{1}(p_0,q_0).
	\end{cases}
	\end{equation}
	It then follows from $S_{12}<0$ and $\delta q_n>0$ that $\omega$ satisfies the bounds
	$$
	-S_{11}(q_0,q_1)<\omega(p_0,q_0)<S_{22}(q_{-1},q_0).
	$$
Thus we can state the following
\begin{lemma}For any symplectic positive twist map $T$ of the cylinder:
	\begin{enumerate}
		\item The set $\mathcal M$ is a closed set invariant under $T$.
		\item The function $\omega: \mathcal M\rightarrow\mathbb R$ is a measurable function satisfying the relations (\ref{relations}).
		\item  The inequalities  $-S_{11}(q_0,q_1)<\omega(p_0,q_0)<S_{22}(q_{-1},q_0) $ are satisfied, for any point $(p_0,q_0)\in\mathcal M$.
			\end{enumerate}
\end{lemma}

In the second  step, we specialize to the assumptions of Theorem \ref{thm:maximizing}.  Since all the orbits of $T$ are assumed to be locally minimizing, we have  $\mathcal M=\Omega$. Moreover, it follows from the bounds $(3)$ of the lemma and the explicit expressions (\ref{Sderivatives}) of the derivatives 
$S_{11}, S_{22}$, that the function $\omega$ is bounded on compact sets and hence can be integrated. Indeed, $S_{11}, S_{22}$ in (\ref{Sderivatives}) are rational functions in $t$ (with no real poles) with the coefficients depending on $r(\phi),r'(\phi),r''(\phi)$.

Now, we choose an arbitrary positive continuous function $A$ and $B:=A\circ T$, we multiply the first and the second equations of (\ref{relations}) by  $B^2$ and $A^2$ respectively, and subtract, getting:
	$$
	B^2\omega(T(p_0,q_0))-A^2\omega(p_0,q_0)=A^2S_{11} +B^2 S_{22}+S_{12}(B^2\delta q_1^{-1}+
	A^2\delta q_1).
	$$

	From $S_{12}<0$ and from the arithmetic-geometric mean inequality, we get:
	$$
	B^2\omega(T(p_0,q_0))-A^2\omega(p_0,q_0)\leq A^2S_{11} +B^2 S_{22}+2ABS_{12}.
	$$

	 Now we can integrate this inequality against the invariant measure over a compact invariant annular region $\Omega_{\beta\gamma}$, which is bounded by $\gamma$ and a rotational invariant curve $\beta$.
	We get the inequality:
	$$
	0\leq\int_{\Omega_{\beta\gamma}}[A^2S_{11} +B^2 S_{22}+2ABS_{12}]d\mu.
	$$

	It is valid for any choice of positive weights $A$ and $B:=A\circ T$. 
	If the integral converges we can take the limit when $\beta$ is chosen approaching the infinity to get 
	$$
	0\leq\int_{\Omega}[A^2S_{11} +B^2 S_{22}+2ABS_{12}]d\mu,
	$$
	completing the proof of Theorem \ref{thm:inequality}.

\end{proof}
	\section{A consequence of Theorem \ref{thm:inequality} }

	In this section, we shall use Theorem \ref{thm:inequality} in order to prove the following:
	\begin{proposition}\label{prop1}
		If all orbits of an outer billiard $\gamma$ are locally minimizing, then the following inequality holds:
		\begin{equation}\label{eq:q}
			\int_{0}^{2\pi}\frac{\sqrt \chi}{r}d\phi\geq2\pi,
		\end{equation}where $r$ is a radial function of $\gamma$ and $
		\chi=r^2(\phi)+2r'^2(\phi)-r(\phi)r''(\phi).
		$
	\end{proposition}

	\begin{proof}

We shall choose the weights $A,B$ as follows:
		$$
		A:=r_0^{-2}=-\frac{1}{2S_1}, \quad B:=r_1^{-2}=\frac{1}{2S_2},
		$$

		With this choice of $A,B$ we compute the integral $I$ of  Theorem \ref{thm:inequality}.
		$$
	I=	\int_{\Omega}[A^2S_{11} +B^2 S_{22}+2ABS_{12}]d\mu=\int_{\Omega}[A^2S_{11} +B^2 S_{22}+2ABS_{12}](-S_{12})d\phi_0d\phi_1=
		$$
		$$=\int_{0}^{2\pi}d\phi\int_{0}^{+\infty}dt[A^2S_{11} +B^2 S_{22}+2ABS_{12}](-S_{12})|J|.
		$$

		Using formulas (\ref{Sderivatives}), we compute:
		\begin{flalign}\label{eq:apart}
		&[A^2S_{11}+2ABS_{12}+B^2S_{22}](-S_{12})J=\\ 
		\nonumber
		&=\frac{2 \chi t^2 \left(\left(1+t^2\right) \mathit{r}^2 \left(-\chi+\mathit{r}^2\right)+\left(\chi t^2+\left(-1+2 t^2\right) \mathit{r}^2\right) \mathit{r}'^2+t^2 \mathit{r}'^4\right)}{\left(\chi t^2+\mathit{r}^2\right) \left(\left(1+t^2\right)^2 \mathit{r}^4+2 t^2 \left(-1+t^2\right) \mathit{r}^2 \mathit{r}'^2+t^4 \mathit{r}'^4\right)}=\\ \nonumber
		&=
		\frac{2 \chi}{\chi t^2+\mathit{r}^2}+\frac{\chi \left(-\mathit{r}+t \mathit{r}'\right)}{\mathit{r} \left(\mathit{r}^2+t^2 \mathit{r}^2-2 t \mathit{r}\mathit{r}'+t^2 \mathit{r}'^2\right)}-\frac{\chi \left(\mathit{r}+t \mathit{r}'\right)}{\mathit{r}\left(\mathit{r}^2+t^2 \mathit{r}^2+2 t \mathit{r} \mathit{r}'+t^2 \mathit{r}'^2\right)}.
		\end{flalign}

		Now, we denote the three summands in the last expression of (\ref{eq:apart}) $F_1,F_2,F_3$. We simplify the integral of (\ref{eq:apart}) by integrating with respect to $t$ every summand:

		\begin{flalign}\label{eq:intgral}
I=&\int_{0}^{2\pi}d\phi\int_{0}^{+\infty}dt[A^2S_{11} +B^2 S_{22}+2ABS_{12}](-S_{12})|J|\\
\nonumber
=&\int_{0}^{2\pi}d\phi\int_{0}^{+\infty}dt
		[F_1+F_2+F_3].
		\end{flalign}
		
		Next, we compute separately the integrals for $F_1$ and $F_2,F_3$.
		$$
		\int_{0}^{2\pi}d\phi \int_{0}^{+\infty}dt [F_1]=\int_{0}^{2\pi}d\phi \left.\frac{2 \sqrt{\chi} \text{ArcTan}\left[\frac{\sqrt{\chi} t}{\mathit{r}(\phi )}\right]}{\mathit{r}(\phi )}\right|_0^{+\infty}=\int_{0}^{2\pi}\frac{\pi\sqrt \chi}{r}d\phi.
		$$

		For $F_2$ and $F_3$, we have:

		\begin{flalign*}
			&\int_{0}^{2\pi}d\phi \int_{0}^{T}dt [F_2]=\\ 
		&=\int_{0}^{2\pi}d\phi\left[\left.\frac{\chi \text{ArcTan}\left[\frac{\mathit{r}'}{\mathit{r}}-\frac{t \left(\mathit{r}^2+\mathit{r}'^2\right)}{\mathit{r}^2}\right]}{\mathit{r}^2+\mathit{r}'^2}+\frac{\chi \text{Log}\left[\mathit{r}^2-2 t \mathit{r} \mathit{r}'+t^2 \left(\mathit{r}
			^2+\mathit{r}'^2\right)\right] \mathit{r}'}{2 \mathit{r} \left(\mathit{r}^2+\mathit{r}'^2\right)}\right]\right|_0^{T},\\
			&\int_{0}^{2\pi}d\phi \int_{0}^{T}dt [F_3]=\\
		&=\int_{0}^{2\pi}d\phi\left[\left.-\frac{\chi \text{ArcTan}\left[\frac{\mathit{r}'}{\mathit{r}}+\frac{t \left(\mathit{r}^2+\mathit{r}'^2\right)}{\mathit{r}^2}\right]}{\mathit{r}^2+\mathit{r}'^2}-\frac{\chi \text{Log}\left[\mathit{r}^2+2 t \mathit{r} \mathit{r}'+t^2 \left(\mathit{r}
			^2+\mathit{r}'^2\right)\right] \mathit{r}'}{2 \mathit{r} \left(\mathit{r}^2+\mathit{r}'^2\right)}\right]\right|_0^{T}.
		\end{flalign*}

		Summing the last two formulas, and then evaluating from $0$ to $T\rightarrow+\infty$, we see that the sum of the terms with $\rm Log$ vanishes and so:
			$$\int_{0}^{2\pi}d\phi \int_{0}^{+\infty}dt [F_2+F_3]=-
		\int_{0}^{2\pi}d\phi\left[\frac{\chi\pi}{r^2+r'^2}\right]
		$$

		The last integral can be computed:
		$$
-	\pi	\int_{0}^{2\pi}d\phi\left[\frac{\chi}{r^2+r'^2}\right]=-\pi\int_{0}^{2\pi}d\phi \sqrt{r^2+r'^2}\frac{\chi}{(r^2+r'^2)^{\frac{3}{2}}}=-\pi\int_{\gamma}k(s)ds=-2\pi^2,
		$$
where we used for the arc-length and for the curvature the expressions: $$ds=d\phi \sqrt{r^2+r'^2}, \ k=\frac{\chi}{(r^2+r'^2)^{\frac{3}{2}}}.$$

		Therefore, altogether we have:
		$$
	I=	\int_{\Omega}[A^2S_{11}+2ABS_{12}+B^2S_{22}]d\mu=\pi\left(\int_{0}^{2\pi}\frac{\sqrt \chi}{r}d\phi-2\pi\right).
		$$
Thus,  Theorem \ref{thm:inequality} yields the inequality (\ref{eq:q}). 
		\end{proof}
		\section{A consequence of the Blaschke-Santalo inequality}

		In this section, we consider an arbitrary simple closed convex  $C^2$ curve $\gamma$ in the plane. Let $\Gamma$ be the convex body bounded by $\gamma$. Using the Blaschke-Santalo inequality, we shall prove the inequality opposite to (\ref {eq:q}), provided that the origin is placed at the Santalo point of $\Gamma$. 
		
		Let me remind the relevant notions.
		For any point in the interior, $x\in Int(\Gamma)$ one defines $\Gamma^x$, which is the polar dual of $\Gamma$ with respect to $x$. By definition, the Santalo point of $\Gamma$ is the unique  point $x\in Int(\Gamma)$, which gives the minimum for the $Area(\Gamma^x)$ (see e.g. \cite{lutwak}).  

Suppose that the convex curve $\gamma$ is such that the origin coincides with the Santalo point of the body $\Gamma$. We denote $\Gamma^*$ the polar dual of $\Gamma$ with respect to the Santalo point. 
Then Blaschke-Santalo inequality states (see e.g. \cite{lutwak}) $$Area(\Gamma)Area(\Gamma^*)\leq\pi^2,$$ with the equality only for ellipses.
		Now we can state the following:
		\begin{proposition}\label{prop2}
			Let $\gamma=\partial \Gamma$ be such that the Santalo point of $\Gamma$ is at the origin. Then 
			\begin{equation}\label{eq:qq}
					\int_{0}^{2\pi}\frac{\sqrt \chi}{r}d\phi\leq2\pi,
				\end{equation}where $r$ is a radial function of $\gamma$ and $
				\chi=r^2(\phi)+2r'^2(\phi)-r(\phi)r''(\phi).
				$ Moreover the equality occurs  if and only if $\gamma$ is an ellipse.
		\end{proposition}
		
		\begin{proof}
		
		Let $h$ be the support function of $\Gamma^*$. Then $h(\phi)=\frac{1}{r(\phi)}$, and we compute:
		$$
		\chi(\phi)=r^2(\phi)+2r'^2(\phi)-r(\phi)r''(\phi)=\frac{h(\phi)+h''(\phi)}{h(\phi)^3}
		$$

		Therefore, we have:
		$$
		\int_{0}^{2\pi}\frac{\sqrt \chi}{r}d\phi=\int_{0}^{2\pi} h\frac{\sqrt{h+h''}}{h^{3/2}}d\phi=\int_{0}^{2\pi}\frac{1}{h}\sqrt{h(h+h'')} d\phi.
		$$

		Then, via the Cauchy-Schwartz inequality we have:
		$$
		\int_{0}^{2\pi}\frac{1}{h}\sqrt{h(h+h'')} d\phi\leq\left( \int_0^{2\pi}\frac{1}{h^2}d\phi\right)^{\frac{1}{2}}\left(\int_{0}^{2\pi}(h^2+hh'')d\phi\right)^{\frac{1}{2}}=
		$$
		$$
	=\left(	\int_0^{2\pi}\frac{1}{h^2}d\phi\right)^{\frac{1}{2}
	}\left(\int_{0}^{2\pi}(h^2-h'^2)d\phi\right)^{\frac{1}{2}}=\sqrt{2Area(\Gamma)}\sqrt{2Area(\Gamma^*)}.
		$$

		 Applying  the Blaschke-Santalo inequality we conclude that:
		 	$$
		 	\int_{0}^{2\pi}\frac{\sqrt \chi}{r}d\phi\leq 2\pi,
		 	$$with the equality only for ellipses.
\end{proof}

\section {Proof of Theorem \ref{thm:maximizing} and Corollary \ref{conjugate}}

		 	Let $\gamma$ be an outer billiard such that all billiard orbits are locally minimizing. Then for an arbitrary choice of the origin in the plane,  Proposition \ref{prop1} implies the inequality (\ref{eq:q}). On the other hand, if 
		 	we choose the origin at the Santalo point, then (\ref{eq:qq}) of Proposition \ref{prop2} gives the opposite inequality. Therefore,  we have the equality in (\ref{eq:qq}) and hence the curve $\gamma$ is an ellipse. This completes the proof of Theorem \ref{thm:maximizing}. \qed
		 	
		 \begin{proof}[Proof of Corrollary \ref{conjugate}]
		 	We need to show that if the curve $\gamma$ is not an ellipse, then there exist a billiard configuration which has conjugate points, i.e a non-trivial Jacobi field vanishing at two points. If this is not the case, then all finite segments of billiard configurations must have non-degenerate Hessian matrices $\delta^2 F_{MN}$. But then all these matrices must be positive definite, by a continuity argument. Therefore, all orbits are locally minimizing and by Theorem \ref{thm:maximizing} the curve $\gamma$ is an ellipse. This contradiction completes the proof of Corollary \ref{conjugate}.
		 	
		 	\end{proof}
		 	
		\smallskip
		
		{\bf Remark.} 
		{\it One can prove by a straightforward computation that the non-standard generating function $S$ and the standard one $H$ satisfy the condition (GA) of \cite{Bialy-Tsodikovich}. This condition guarantees that the classes $\mathcal M_S, \mathcal M_H$
		of locally minimizing orbits for $S$ and $H$ coincide. This fact implies in particular, that if the curve is not an ellipse then the conjugate points necessarily  exist also for the functional corresponding to $H$.
		We will not dwell on this in this paper.}

\end{document}